\documentclass[letterpaper, reqno,11pt]{article}
\usepackage[margin=1.0in]{geometry}
\usepackage{color,latexsym,amsmath,amssymb, amsthm, graphicx,subfigure, mathalfa}
\usepackage[percent]{overpic}
\usepackage{todonotes}
\usepackage{hyperref}
\usepackage{comment}
\usepackage{enumitem}

\newcommand{\RR}{\mathbb{R}}

\newcommand{\p}{{\bf p}}
\newcommand{\dd}{{\bf d}}

\newtheorem{thm}{Theorem}
\newtheorem{lem}[thm]{Lemma}
\newtheorem{prop}[thm]{Proposition}

\newtheorem{conj}[thm]{Conjecture}
\theoremstyle{definition}
\newtheorem{defn}[thm]{Definition}
\theoremstyle{definition}
\newtheorem*{remark}{Remark}
\newtheorem*{remarks}{Remarks}
\newtheorem{example}[thm]{Example}

\usepackage{authblk}
\title{On the $k$-volume rigidity of a simplicial complex in $\mathbb{R}^d$}

\author[2]{Alan Lew\thanks{\href{mailto:alanlew@andrew.cmu.edu}{alanlew@andrew.cmu.edu}. 
}}
\author[1]{Eran Nevo\thanks{\href{mailto:nevo@math.huji.ac.il}{nevo@math.huji.ac.il}. 
E. Nevo was partially supported by the Israel Science Foundation grants ISF-2480/20 and ISF-687/24.
}}
\author[1]{Yuval Peled\thanks{\href{mailto:yuval.peled@mail.huji.ac.il}{yuval.peled@mail.huji.ac.il}. 
Y. Peled was partially supported by the Israel Science Foundation grant ISF-3464/24.
}}
\author[1,3]{Orit E. Raz\thanks{\href{mailto:oritraz@mail.huji.ac.il}{oritraz@mail.huji.ac.il}. O.\;E. Raz was partially supported by the Charles Simonyi Endowment.}}

\affil[1]{Einstein Institute of Mathematics,
 Hebrew University, Jerusalem~91904, Israel}
 \affil[2]{Dept. Math. Sciences, Carnegie Mellon University, Pittsburgh, PA 15213, USA}
\affil[3]{Institute for Advanced Study, Princeton, NJ 08540, USA}
 	
	\date{}
\begin{document}
\maketitle
\begin{abstract}
We define a generic rigidity matroid for $k$-volumes of a simplicial complex in $\RR^d$, and prove that for $2\leq k \leq d-1$ it has the same rank as the classical generic $d$-rigidity matroid on the same vertex set (namely, the case $k=1$). This is in contrast with the $k=d$ case, previously studied by Lubetzky and Peled, which presents a different behavior. We conjecture a characterization for the bases of this matroid in terms of $d$-rigidity of the $1$-skeleton of the complex and a combinatorial Hall condition on incidences of edges in $k$-faces.   

\end{abstract}

\section{Introduction}

A simplicial complex $X$ is a family of subsets of some finite set $V$, such that if $\sigma\in X$, then $\tau\in X$ for all $\tau\subset \sigma$. The elements of $X$ are called the \emph{simplices} or \emph{faces} of $X$. The \emph{dimension} of a face $\sigma\in X$ is defined as $|\sigma|-1$. The dimension of the complex $X$ is the maximum dimension of a simplex in $X$. For a $k$-dimensional simplicial complex $X$ and $i=0,1,\ldots,k$, we denote by $X_i$ the set of $i$-dimensional faces of $X$. The set $X_0$ of $0$-dimensional faces of $X$ is called the \emph{vertex set} of $X$.

Let $X$ be a $k$-dimensional simplicial complex, and let $\p\in (\mathbb{R}^d)^{|X_0|}$ be an embedding of its vertex set $X_0$ in $\mathbb{R}^d$. For $0\le i\le k$ and $\tau\in X_i$, we denote by $\p(\tau)$ the convex hull of the image of $\tau$ under $\p$, and by ${\rm vol}_{i}(\tau)$ its $i$-dimensional volume. 

Let $v_X:(\RR^d)^{|X_0|}\to\RR^{|X_k|}$ map an embedding of $X_0$ in $\RR^d$ to the vector of squared $k$-volumes of its $k$-dimensional simplices. 
That is, for $\p\in (\RR^d)^{|X_0|}$, $v_X(\p)\in\RR^{|X_k|}$ is defined by
\[
    v_X(\p)_{\sigma}= {\rm vol}_k(\sigma)^2
\]
for all $\sigma\in X_k$. We define the {\it $(k,d)$-volume rigidity matrix} of $X$ at an embedding $\p\in (\RR^d)^{|X_0|}$ to be 
$${\cal B}(X,\p):=J_{v_X}(\p),$$
the Jacobian matrix of $v_X$ at the point $\p$, which is an $|X_k|\times d|X_0|$ matrix.

We index the rows of ${\cal B}(X,\p)$ by simplices in $X_k$, and every $d$ consecutive columns of ${\cal B}(X,\p)$ by the vertices in $X_0$. Using the fact that for every  $\sigma\in X_k$ and $v\in \sigma$, ${\rm vol}_k(\sigma)={\rm vol}_{k-1}(\sigma\setminus\{v\})\cdot h/k!$, where $h$ is the altitude of $\p(\sigma)$ with respect to $\p(v)$, it is easy to check that
\begin{equation}\label{Bexplicit}
{\cal B}(X,\p)_{\sigma,v}=
\begin{cases}
\frac{2 {\rm vol}_k(\sigma)}{k!}\cdot {\rm vol}_{k-1}(\sigma\setminus\{v\}) N_{\sigma,v}& \quad v\in \sigma,\\
0& \quad \text{otherwise},
\end{cases}
\end{equation}
for all $\sigma\in X_k$ and $v\in X_0$, where $N_{\sigma,v}\in\mathbb{R}^d$ stands for the unit vector parallel to the altitude of $\p(\sigma)$ with respect to $\p(v)$ (pointing from the $(k-1)$-flat spanned by $\p(\sigma\setminus \{v\})$ to $\p(v)$).

We define the \emph{$(k,d)$-volume rigidity matroid} of $X$, denoted by ${\cal M}_{k,d}(X)$, to be the matroid whose elements are the $k$-simplices of $X$ and its independent sets 
correspond to linearly independent sets of rows of the matrix ${\cal B}(X,\p)$, for a generic\footnote{By generic we mean that the $d|X_0|$ entries defining $\p$ are algebraically independent over the field of rationals.} $\p$. Note that ${\cal M}_{1,d}(X)$ is the standard $d$-rigidity matroid of the graph $G=(X_0,X_1)$; 
see e.g.~\cite{book:Connelley-Simon,book:Graver-Servatius-Servatius,chapter:Jordan,AR78,AR79} for relevant background.  The matroid ${\cal M}_{d,d}(X)$, corresponding to the case $k=d$, 
was introduced in~\cite[Appendix A]{lubetzky2023threshold} (see also \cite{borcea2013realizations}) and further  
 studied in \cite{BulavkaNevoPeled}.

Let $\Delta_{n,k}$ denote the complete $k$-dimensional simplicial complex on $n$ vertices. We introduce the following definition of $k$-volume rigidity in $\RR^d$.

\begin{defn}
Let $X$ be a $k$-dimensional simplicial complex on $n$ vertices, and let $d\ge k$. We say that $X$ is \emph{$k$-volume rigid in $\RR^d$} if 
$${\rm rank}({\cal M}_{k,d}(X))={\rm rank}({\cal M}_{k,d}(\Delta_{n,k})).$$ 
\end{defn}

The following is our main result.

\begin{thm}\label{main1}
Let $d\ge 2$. Let $k,n$ such that either 
(i) $k=d-1$ and  $n\ge d+2$, or (ii) $1\le k\le d-2$ and $n\ge d+1$.
Then,
$$
{\rm rank}({\cal M}_{k,d}(\Delta_{n,k}))=dn-\binom{d+1}{2}.
$$
\end{thm}

The proof of Theorem \ref{main1} consists of two main steps. First, we apply an inductive argument to reduce the problem to the case $n=d+2$ (when $k=d-1$) or $n=d+1$ (when $1\le k\le d-2$). Then, we analyze these base cases by showing, in each case, that the corresponding $(k,d)$-volume rigidity matrix (for a specific choice of embedding $\p$) is tightly related to certain ``subset inclusion matrix", introduced by Goettlieb in \cite{gottlieb1966certain} (and independently by Graver and Jurkat in \cite{graver1973module}, and by Wilson in \cite{wilson1973necessary}).

\begin{remarks}\,
\begin{enumerate}[leftmargin=*]
\item The case $k=1$ is a well-known result about the standard $d$-rigidity matroid (see, e.g, \cite{chapter:Jordan}).
The case $k=d$ behaves differently. Indeed, it was shown by Lubetzky and Peled~\cite[Appendix A]{lubetzky2023threshold} that, for all $n\ge d+1$, 
$${\rm rank}({\cal M}_{d,d}(\Delta_{n,d}))=dn-(d^2+d-1).$$

\item The case $k=d-1$ and $n=d+1$, excluded from Theorem \ref{main1}, follows by similar arguments. Indeed, it is easy to show that in this case we have ${\rm rank}({\cal M}_{d-1,d}(\Delta_{d+1,d-1}))=d+1$ (see Proposition \ref{prop:k=d-1,n=d+1}).

\item Let us mention that a related, but different, notion of rigidity for simplicial complexes was studied by Lee in \cite{lee1996pl} (building on previous unpublished work by Filliman), and further developed by Tay, White, and Whiteley  in \cite{tay1995skeletal,tay1995skeletal2} (see also \cite{tay2000homological}).

\end{enumerate}
\end{remarks}

The paper is organized as follows. In Section \ref{sec:prelim} we prove some preliminary results that will be used later. In Section \ref{sec:main} we present the proof of our main result, Theorem \ref{main1}. In Section \ref{sec:discussion} we propose a conjecture providing a characterization for the rank of the $(k,d)$-volume rigidity matroid of a complex in terms of the standard rigidity matroid of its $1$-skeleton, and discuss some of its consequences.

\section{Preliminary results}\label{sec:prelim}
\subsection{Adding a vertex}

Let $X$ be a simplicial complex, and let $v\in X_0$. We denote by $X\setminus v$ the simplicial complex on vertex set $X_0\setminus\{v\}$ whose simplices are all the simplices of $X$ that do not contain $v$. We define the \emph{link} of $v$ in $X$ to be the subcomplex of $X$ consisting of all simplices of the form $\sigma\setminus \{v\}$, where $v\in \sigma\in X$.

\begin{lem}[Vertex addition  lemma]\label{lem:addvertex}
Let $1\le k\le d-1$ and let $n\ge d+1$. Let $X$ be a $k$-dimensional simplicial complex on $n$ vertices and let $v\in X_0$ be a vertex 
whose link in $X$ is the complete $(k-1)$-dimensional complex on $X_0\setminus \{v\}$. 
Then, 
$$
{\rm rank}({\cal M}_{k,d}(X)) \ge {\rm rank}({\cal M}_{k,d}(X\setminus v))+d.$$
\end{lem}
\begin{proof}
Let $\p$ be a generic embedding of $X_0$ in $\RR^d$. Consider the matrices ${\cal B}={\cal B}(X,\p)$ and 
${\cal B}'={\cal B}(X\setminus v, \p|_{X_0\setminus \{v\}})$. 
Our goal is to show that     
$$
{\rm rank}({\cal B})\ge{\rm rank}({\cal B}')+d.
$$ 
Note that $|X_0\setminus \{v\}|\ge d$, and let $v_1,\ldots,v_d\in X_0\setminus \{v\}$ be distinct vertices. 
Let
\[
    \Sigma= \{ \eta\cup\{v\} :\, \eta\subset\{v_1,\ldots,v_d\},\,|\eta|=k\}.
\]
That is, $\Sigma$ is the set of $k$-simplices that are the union of $v$ with a $k$-subset of $\{v_1,\ldots,v_d\}$.
Let $\hat X$ be the 
subcomplex of $X$ defined by $\hat{X}_i=X_i$ for $0\le i\le k-1$, and 
$$\hat X_k:=(X\setminus v)_k\cup \Sigma.$$
Let $\hat{\cal B}= {\cal B}(\hat{X},\p)$. As $\hat{X}_k\subset X_k$ and $\hat{X}_0 = X_0$, we obtain $\hat{\cal B}$ from ${\cal B}$ by removing some of its rows. Hence,
\[
    {\rm rank}({\cal B})\ge{\rm rank}(\hat{{\cal B}}).
\]
Fixing an ordering on the rows of $\hat{\cal B}$ so that the rows corresponding to the $k$-simplices in $\Sigma$ are last, and an ordering on its columns so the $d$ columns corresponding to the vertex $v$ are last, we get that 
$\hat{\cal B}$ is a block lower-triangular matrix, consisting of two diagonal blocks: one of them is ${\cal B}'$, and the other, which we denote by ${\cal B}''$, 
is 
the $\binom{d}{k}\times d$ submatrix of ${\cal B}$ corresponding to the rows of $\Sigma$ and to the $d$ columns associated with $v$. Note that, by our assumption on $k$, we have that $\binom{d}{k}\ge d$.

We claim that ${\rm rank}({\cal B}'')=d$. 
Indeed, in view of \eqref{Bexplicit}, the rows of ${\cal B}''$ correspond to a scaling of the vectors $N_{\sigma,v}(\p)$, for $\sigma\in \Sigma$. 
For $\sigma\in \Sigma$, let $y_{\sigma}(\p)$ be the foot of the altitude of $\p(\sigma)$ with respect to $\p(v)$. Note that the vectors $\{N_{\sigma,v}(\p)\}_{\sigma\in \Sigma}$ linearly span $\RR^d$ if and only if the affine span of $\{y_{\sigma}(\p)\}_{\sigma\in \Sigma}$ in $\RR^d$ has dimension $d-1$.
Now, consider a special embedding $\p'\in(\RR^d)^{|X_0|}$ that maps the vertices $v,v_1,\ldots,v_d$ to the vertices of a regular $d$-simplex in $\RR^d$. Observe that for each $\sigma\in\Sigma$, $y_{\sigma}(\p')$ is the barycenter of $\p'(\sigma\setminus\{v\})$. It is easy to see that the convex hull of $\{y_{\sigma}(\p')\}_{\sigma\in \Sigma}$  has dimension $d-1$, thus the vectors $\{N_{\sigma,v}(\p')\}_{\sigma\in \Sigma}$ linearly span $\RR^d$.
Finally, the entries of ${\cal B}''$ are algebraic expressions in the entries of the embedding $\p$, 
and thus,
as $\p$ is generic, the vectors $\{N_{\sigma,v}(\p)\}_{\sigma\in\Sigma}$ also linearly span $\RR^d$, or equivalently, ${\rm rank}({\cal B}'')=d$. 
Thus, we obtain
\[
    {\rm rank}({\cal B})\ge{\rm rank}(\hat{{\cal B}}) \ge {\rm rank}({\cal B}') +d,
\]
as wanted.
\end{proof}

\subsection{Cayley--Menger formula}
Consider a $k$-dimensional simplex, $\sigma$, with vertices $\{0,1,\ldots,k\}$, embedded in $\RR^d$ (for some $d\ge k$). For $0\le i<j\le k$, let $d_{ij}$ denote the distance between the vertices $i$ and $j$. Recall that by the Cayley--Menger formula (see, for example, \cite[Chapter 1]{fiedlerbook}), we have
\begin{equation}\label{CMformula}
{\rm vol}_k^2(\sigma)= g(d_{01},d_{02},\ldots,d_{k-1,k}):= 
\frac{(-1)^{k+1}}{(k!)^22^k}\det\begin{bmatrix}
    0       & d_{01}^2 & d_{02}^2 & \dots & d_{0k}^2&1 \\
    d_{01}^2       &0 & d_{12}^2 & \dots & d_{1k}^2&1 \\
    d_{02}^2       &d_{12}^2 & 0 & \dots & d_{2k}^2&1 \\
    \vdots       &\vdots & \vdots & \ddots & \vdots&\vdots \\
    d_{0k}^2       &d_{1k}^2 & d_{2k}^2 & \dots & 0&1 \\
  1       &1 & 1 & \dots & 1&0 
    \end{bmatrix}.
    \end{equation}

\begin{lem}\label{lem:nonzero}
Let $f(t)=g(\sqrt{t},d_{02},\ldots,d_{k-1,k})$ denote the squared $k$-volume of a $k$-dimensional simplex on $\{0,1,\ldots,k\}$  whose edge lengths $d_{ij}>0$ are fixed, except for the edge $\{0,1\}$ whose squared length is a parameter $t$. Let $t_0,t_{\pi/2},t_{\pi}$ be the values of $t$ for which the angle between the simplices $\{0\}\cup \{2,\ldots,k\}$ and $\{1\}\cup\{2,\ldots,k\}$ is $0$, $\pi/2$, and $\pi$, respectively. Note that the domain of $f$ is the closed interval $[t_0,t_{\pi}]$. Then, for $t\in(t_0,t_{\pi})$, $f'(t)\ne 0$ unless $t=t_{\pi/2}$.
\end{lem}
\begin{proof}
Using the Cayley--Menger formula \eqref{CMformula} and setting $t=d_{01}^2$, we get
$$
f(t)=
At^2+Bt+C,
$$
for some constants $A,B,C$. 
It is easy to verify that
$$
A=
\begin{cases}
\tfrac{-1}{16}& \quad k=2,\\
\tfrac{-1}{4k^2(k-1)^2}{\rm vol}_{k-2}^2(\sigma_{01})&\quad k>2,\end{cases}
$$
where $\sigma_{01}$ stands for the $(k-2)$-dimensional simplex spanned by $\{2,\ldots,k\}$. In particular, $A\neq 0$. 

Thus, 
$f'(t)=2At+B$, and we have $f'(t)=0$ for $t=\frac{-B}{2A}$. In particular, $f(t)$ has at most one local extremum. On the other hand, we know that $f$ has local maximum when the $(k-1)$-dimensional simplices spanned by $\{0\}\cup \{2,\ldots,k\}$ and $\{1\}\cup\{2,\ldots,k\}$ are orthogonal to one another, that is, when $t=t_{\pi/2}$. The lemma then follows.
\end{proof}

\begin{remark}
In the proof of Lemma~\ref{lem:nonzero}, it follows that
$B=-2A(d_{01}^*)^2$
and $C=A(d_{01}^*)^4+(v^*)^2$
where $d_{01}^*$ is the edge-length when the volume is maximized (that is, when the two faces $\{0\}\cup \{2,\ldots,k\}$ and $\{1\}\cup\{2,\ldots,k\}$ are orthogonal), and $v^*$ is this maximal volume.
However, we do not use these facts.
\end{remark}

\begin{defn}Let $X$ be a $k$-dimensional simplicial complex. Let $h_X:\RR^{|X_1|}\to\RR^{|X_k|}$ map squared edge lengths to squared $k$-volumes, using the Cayley--Menger formula \eqref{CMformula}. 
Let $${\cal C}(X,\dd):=J_{h_{X}}(\dd)$$ denote the Jacobian matrix of $h_X$, evaluated for a given vector $\dd$ of squared edge lengths, arising from some embedding $\p$ of $X_0$ in $\RR^d$. 
\end{defn}

Note that ${\cal C}(X,\dd)$ is an $|X_k|\times |X_1|$ matrix, whose $(\sigma,e)$-entry, for $\sigma\in X_k$ and $e\in X_1$, is the partial derivative of the square of the $k$-volume of the simplex $\sigma$  with respect to the variable $\dd_e$, evaluated at the point $\dd$. 
Observe that  
$$v_X=h_{X}\circ f_{X_1},$$
where $f_{X_1}:(\RR^{d})^{|X_0|}\to \RR^{|X_1|}$ is the squared edge length map. That is, $f_{X_1}$ is the map defined by
\[
    f_{X_1}(\p)_e= \|\p(u)-\p(v)\|^2
\]
for all $e=\{u,v\}\in X_1$. By the chain rule, we have
\begin{equation}\label{chainrule}
{\cal B}(X,\p)={\cal C}(X,\dd) \cdot R(G,\p),\end{equation}
where $R(G,\p)=J_{f_{X_1}}(\p)$ is the standard rigidity matrix of the graph $G=(X_0,X_1)$
and  $\dd:=f_{X_1}(\p)$.

\section{Proof of main result}\label{sec:main}

Our main result, Theorem \ref{main1}, follows immediately from the next two statements.

\begin{prop}\label{main_part1}
Let $d\ge 2$, $k=d-1$, and $n\ge d+2$. Then, 
$$
{\rm rank}({\cal M}_{k,d}(\Delta_{n,k}))=dn-\binom{d+1}{2}.
$$
\end{prop}

\begin{prop}\label{main_part2}
Let $d\ge 3$, $1\le k\le d-2$, and $n\ge d+1$. Then, 
$$
{\rm rank}({\cal M}_{k,d}(\Delta_{n,k}))=dn-\binom{d+1}{2}.
$$
\end{prop}

\begin{proof}[Proof of Proposition~\ref{main_part1}]
First note that, for any $n\ge d+1$, we have 
\begin{equation}\label{lowerbound}
{\rm rank}({\cal M}_{k,d}(\Delta_{n,k}))\le dn-\binom{d+1}{2}.
\end{equation}
Indeed, by \eqref{chainrule}, we have, for any $\p$,
\[
{\rm rank}({\cal B}(\Delta_{n,k},\p))\le {\rm rank}(R(\Delta_{n,1},\p))\le dn-\binom{d+1}{2},
\]
where the last inequality is a standard result on the rigidity of graphs (see e.g. \cite[Lemma 1.2.1]{chapter:Jordan}). Hence, the inequality \eqref{lowerbound} follows. 

Note also that it suffices to prove the theorem for $n=d+2$. That is, it suffices to prove that
\begin{equation}\label{casen=d+2}
    {\rm rank}({\cal M}_{k,d}(\Delta_{d+2,k}))=d(d+2)-\binom{d+1}{2}
.\end{equation}
Indeed, given \eqref{casen=d+2}, we can then repeatedly apply Lemma~\ref{lem:addvertex} to add the $n-(d+2)$ remaining vertices of $\Delta_{n,k}$ one by one. At each step, when we add a vertex $v$, the rank of the resulting matrix is increased by at least $d$,
hence  
$${\rm rank}({\cal M}_{k,d}(\Delta_{n,k})) \ge
d(d+2)-\binom{d+1}{2}+d(n-(d+2))=dn-\binom{d+1}{2}.$$
Together with the reverse inequality \eqref{lowerbound}, this proves the theorem.

Thus, we only need to prove \eqref{casen=d+2}.
Recall that for a matrix whose entries are algebraic expressions in the entries of $\p$, its generic rank is also its maximal one. So, to prove \eqref{casen=d+2}, it suffices to find an embedding $\p$ for which
\begin{equation}\label{d+2points}
{\rm rank}({\cal B}(\Delta_{d+2,k},\p))
=d(d+2)- \binom{d+1}{2}.
\end{equation}

Write $X=\Delta_{d+2,k}$. For convenience, assume $X_0=[d+2]$. Let $\p$ be an embedding of $X_0$ in $\RR^{d}$ for which $\p(2),\ldots,\p(d+2)$ are the vertices of a regular $d$-simplex, and $\p(1)$ is the centroid of this simplex.  
We claim that \eqref{d+2points} holds for this choice of $\p$.

Using \eqref{chainrule}, we may write
$$
{\cal B}(\Delta_{d+2,k},\p)
={\cal C}(\Delta_{d+2,k},\dd) \cdot R(K_{d+2},\p),
$$
where $K_{d+2}$ stands for the complete graph on $d+2$ vertices, and $\dd$ is the vector of squared edge lengths induced by $\p$.
It is well-known and easy to check that
$$
{\rm rank}(R(K_{d+2},\p))=d(d+2)-\binom{d+1}{2}.
$$
Note also that, as $k=d-1$ and $n=d+2$, we have 
$$
|X_1|=|X_k|=\binom{d+2}{2},$$
and so ${\cal C}(\Delta_{d+2,k},\dd)$ is a $\binom{d+2}{2}\times \binom{d+2}{2}$ square matrix.
Thus, to prove \eqref{d+2points}, it suffices to prove that 
\begin{equation}
\text{${\cal C}={\cal C}(\Delta_{d+2,k},\dd)$ is invertible.}
\end{equation}
Index the rows of ${\cal C}$ by the elements of $X_k$ and its columns by the elements of $X_1$. 
Observe that, for $e\in X_1$ and $T\in X_{k}$, the $(T,e)$-entry of ${\cal C}$ is $0$ if $e\not\subset T$. Moreover, by symmetry, we have
$$
{\cal C}_{T,e}=
\begin{cases}
0& \quad e\not\subset T,\\
\alpha&\quad e\subset T,\;\;1\not\in T,\\
\beta&\quad e\subset T,\;\;1\in T\cap e,\\
\gamma&\quad e\subset T,\;\;1\in T\setminus e,
\end{cases}
$$
for some three real numbers $\alpha, \beta, \gamma$.
By Lemma~\ref{lem:nonzero}, we have that $\alpha$, $\beta$, and $\gamma$ are non-zero.

We apply the following row- and column- scaling: for every $T\in X_k$, we multiply the $T$-row by $\frac 1\alpha$ if $1\not\in T$, and by $\frac{1}{\gamma}$ otherwise. Then, for every $e\in X_1$ such that $1\in e$, we multiply the $e$-column by $\frac{\gamma}{\beta}$.  
The resulting matrix, which we denote by $A_{k+1,2}^{d+2}$, satisfies
\[
    \left(A_{k+1,2}^{d+2}\right)_{T,e}=\begin{cases}
    0 & e\not\subset T,\\
    1 & e\subset T,
    \end{cases}
\]
for all $e\in X_1$ and $T\in X_k$. That is, $A_{k+1,2}^{d+2}$ is the incidence matrix between $(k+1)$-subsets and $2$-subsets of $X_0=[d+2]$. By Gottlieb~\cite[Corollary 2]{gottlieb1966certain} (or alternatively, by \cite{graver1973module,wilson1973necessary}), $A_{k+1,2}^{d+2}$ is invertible. Since $A_{k+1,2}^{d+2}$ was obtained from ${\cal C}$ by row and column operations,  ${\cal C}$ is invertible as well. This completes the proof of the proposition.

\end{proof}

\begin{proof}[Proof of Proposition~\ref{main_part2}]
By an argument similar to the one at the beginning of the proof of Proposition~\ref{main_part1}, it suffices to prove the proposition for $n=d+1$. Indeed, for $n>d+1$, we can then add the remaining $n-(d+1)$ vertices one by one, using Lemma~\ref{lem:addvertex}.

Write $X=\Delta_{d+1,k}$, and assume for convenience $X_0=[d+1]$. Let $\p$ be an embedding of $X_0$ in $\RR^{d}$ for which $\p(1),\ldots,\p(d+1)$ are the vertices of a regular $d$-simplex.
We claim that, for $1\le k\le d-2$, one has 
\begin{equation}\label{rankcompeasy}
{\rm rank}({\cal B}(\Delta_{d+1,k},\p))=d(d+1)-\binom{d+1}{2}.
\end{equation}
Using \eqref{chainrule}, we may write
\begin{equation}\label{decompBd+1}
{\cal B}(\Delta_{d+1,k},\p)
={\cal C}(\Delta_{d+1,k},\dd) \cdot R(K_{d+1},\p),
\end{equation}
where $K_{d+1}$ stands for  the complete graph on $d+1$ vertices, and $\dd$ is the vector of squared edge lengths induced by $\p$. Namely, $\dd$ is the all-ones vector. Note that 
$$
{\rm rank}(R(K_{d+1},\p))=d(d+1)-\binom{d+1}{2}=\binom{d+1}{2}.
$$
Since the number of rows of $R(K_{d+1},\p)$ is exactly $\binom{d+1}{2}$,  $R(K_{d+1},\p)$ has a full rank, and its image is all of $\RR^{\binom{d+1}{2}}$.
In view of \eqref{decompBd+1}, this implies that 
$$
{\rm rank}({\cal B}(\Delta_{d+1,k},\p))
={\rm rank}({\cal C}(\Delta_{d+1,k},\dd)).
$$
So, in order to prove \eqref{rankcompeasy}, we  need to show that 
\begin{equation}\label{rankC}
{\rm rank}({\cal C}(\Delta_{d+1,k},\dd))=\binom{d+1}{2}.
\end{equation}
Write ${\cal C}={\cal C}(\Delta_{d+1,k},\dd)$. We  index the rows of ${\cal C}$ 
by the elements of $X_k$ and the columns of ${\cal C}$ by the elements of $X_1$. 
Observe that for $e\in X_1$ and $T\in X_{k}$, the $(T,e)$-entry of ${\cal C}$ is $0$ if $e\not\subset T$. Moreover, by symmetry, we have
$$
{\cal C}_{T,e}=
\begin{cases}
0& \quad e\not\subset T,\\
\alpha&\quad e\subset T,
\end{cases}
$$
for some real number $\alpha$.
By Lemma~\ref{lem:nonzero}, we have  $\alpha\neq 0$.

Thus ${\cal C}=\alpha A_{k+1,2}^{d+1}$, where $A_{k+1,2}^{d+1}$ is the incidence matrix between $(k+1)$-subsets and $2$-subsets of $X_0=[d+1]$. By ~\cite[Corollary 2]{gottlieb1966certain}, the matrix $A_{k+1,2}^{d+1}$ has maximal rank. Since $1\le k\le d-2$, we get
$$
{\rm rank}({\cal C})
=
{\rm rank}(A_{k+1,2}^{d+1})
=\binom{d+1}{2},$$ as needed. This proves \eqref{rankC}, and hence \eqref{rankcompeasy}, and thus 
completes the proof of the proposition.
\end{proof}

Finally, let us note that the case $k=d-1$ and $n=d+1$, excluded from Theorem \ref{main1}, follows easily by similar arguments, as detailed next.

\begin{prop}\label{prop:k=d-1,n=d+1}
    Let $d\ge 2$. Then,
    \[
        {\rm rank}({\cal M}_{d-1,d}(\Delta_{d+1,d-1}))=d+1.
    \]
\end{prop}
\begin{proof}
Let $\p$ map $(\Delta_{d+1,d-1})_{0}$ to the vertices of a regular $d$-simplex in $\mathbb{R}^d$. Then, following the same arguments as in Proposition \ref{main_part2}, we obtain
\[
    {\rm rank}(\mathcal{B}(\Delta_{d+1,d-1},\p))= {\rm rank}(A_{d,2}^{d+1}),
\]
where $A_{d,2}^{d+1}$ is the incidence matrix between $d$-subsets and $2$-subsets of the set $[d+1]$. By~\cite[Corollary 2]{gottlieb1966certain}, this matrix has maximal rank, namely ${\rm rank}(A_{d,2}^{d+1})=d+1$. Therefore, 
\[
{\rm rank}({\cal M}_{d-1,d}(\Delta_{d+1,d-1}))\ge {\rm rank}(\mathcal{B}(\Delta_{d+1,d-1},\p))=d+1.
\]
On the other hand, since $\mathcal{M}_{d-1,d}$ has exactly $\binom{d+1}{d}=d+1$ elements, we must have
\[
{\rm rank}({\cal M}_{d-1,d}(\Delta_{d+1,d-1}))=d+1,\]
as wanted.
\end{proof}

\section{Discussion}\label{sec:discussion}
For a complex $X$, we call the graph $G=(X_0,X_1)$ the \emph{$1$-skeleton} of $X$. Combining (\ref{chainrule}) with Theorem~\ref{main1}, we obtain that if a simplicial complex on $n\ge d+2$ vertices is $k$-volume rigid in $\RR^d$, for some $d>k\ge 1$, then its $1$-skeleton must be $d$-rigid as a graph. The converse does not hold (see Example \ref{ex1} below). However, we propose the following conjecture, which gives a characterization for the rank of the $(k,d)$-volume rigidity matroid of a complex in terms of the standard $d$-rigidity matroid of its $1$-skeleton.

For two disjoint families of sets ${\cal A},{\cal B}$, let $H_{{\cal A},{\cal B}}$ be the bipartite graph on vertex set ${\cal A}\cup {\cal B}$ with edge set $\{\{A,B\}:\, A\in {\cal A},\, B\in {\cal B},\, A\subset B\}$. For a graph $G=(V,E)$, let $\nu(G)$ be its matching number, that is, the size of a maximum matching in $G$.

\begin{conj}\label{conj}
Let $d\ge 3$, $1\le k\le d-1$, and let $X$ be a $k$-dimensional simplicial complex on $n\ge d+2$ vertices. 
Then, 
\[
    {\rm rank}({\cal M}_{k,d}(X))= \max_{E} \nu(H_{E,X_k}),
\]
where the maximum is taken over all $E\subset X_1$ that are independent in the standard $d$-rigidity matroid.
\end{conj}

In particular, if $|X_k|=dn-\binom{d+1}{2}$, then Conjecture \ref{conj} implies that $X$ is $k$-volume rigid in $\RR^d$ if and only if there exists $E\subset X_1$ of size $dn-\binom{d+1}{2}$ such that 
$G=(X_0,E)$ is minimally rigid in $\RR^d$, and
there exists a perfect matching between
$E$ and $X_k$ in the bipartite incidence graph $H_{E,X_k}$. 
By a classical result of Rado (\cite{rado1942theorem}; see also \cite{welsh1970matroid,welsh1971genHall}), this is equivalent to the following statement.

\begin{conj}\label{conj2}
Let $d\ge 3$, $1\le k\le d-1$, and 
let $X$ be a $k$-dimensional simplicial complex on $n\ge d+2$ vertices. Assume that $|X_k|=dn-\binom{d+1}{2}$. 
Then, $X$ is $k$-volume rigid in $\RR^d$ if and only if for every $S\subset X_k$, the 1-skeleton of the restriction $X[S]=\{\tau\in X:\, \tau\subset \sigma \text{ for some } \sigma\in S\}$ has rank at least $|S|$ in the standard $d$-rigidity matroid. 
\end{conj}
Note that the ``only if" direction of the conjecture holds. 
Indeed, for a subset $S\subset X_k$ and a generic $\p$, consider the $|S|\times d|X_0|$ submatrix $Q$ of ${\cal B}(X,\p)$ corresponding to the rows of $S$. On the one hand, $\left({\cal C}(X,f_{X_1}(\p))\right)_{\sigma,e}=0$ for every edge $e\notin X[S]$ and every $k$-simplex $\sigma\in S$. Therefore, using \eqref{chainrule}, we find that $Q$ is a product of a submatrix of ${\cal C}(X,f_{X_1}(\p)))$ and the standard $d$-rigidity matrix of the $1$-skeleton of $X[S]$. On the other hand, assuming that $|X_k|=dn-\binom{d+1}{2}$ and $X$ is $k$-volume rigid in $\RR^d$, we have ${\rm rank}(Q)=|S|$. The ``only if" direction follows since matrix multiplication does not increase the rank.

\begin{example}\label{ex1}
Let $Y$ be the simplicial complex obtained from the complete $2$-dimensional complex on $5$ vertices by removing a single triangle. Let $Z$ be obtained by gluing together two copies of $Y$ along an edge. Any such $Z$ has 8 vertices and 18 triangles. Note that the graph $G=(Z_0,Z_1)$ is not generically rigid in $\RR^3$, as one can fix one copy of $Y$ and rotate the other copy along the common edge. The same motion also shows that $Z$ is not $2$-volume rigid in $\mathbb R^3$. Let $a,b,c$ be the vertices in $Z$ unique to the first copy of $Y$, and let $a',b',c'$ be the vertices in $Z$ unique to the second copy. Let $v$ be a new vertex not in $Z$, and let $X$ be the union of $Z$ and the 3 triangles $\{a,a',v\}, \{b,b',v\}$, and $\{c,c',v\}$. Let $\p$ be a generic embedding of $X_0$ in $\RR^3$, and let $G'=(X_0,X_1)$. Then, it is not hard to verify that ${\rm rank}(R(G',\p))=21={\rm rank}(
{\cal C}(X,f_{X_1}(\p)))$, but ${\rm rank}({\cal B} (X,\p))=20$, so while the $1$-skeleton of $X$ is generically rigid in $\RR^3$, $X$ is not 2-volume rigid in $\RR^3$. The Hall condition mentioned above indeed fails here: letting $S$ be the collection of triangles in $X$ not containing $v$, it is easy to check that the $1$-skeleton of $X[S]$, which is the graph $G=(Z_0,Z_1)$, has rank 17 in the standard $3$-rigidity matroid, but $|S|=18>17$.   
\end{example}

It is natural to wonder about the rank of the matrix ${\cal C}(X,\dd)$ for a generic vector $\dd\in\RR^{|X_1|}$. The following is a special case of Conjecture~\ref{conj}.
\begin{conj}
Let $2\le k\le n-2$, and let $X$ be a $k$-dimensional simplicial complex on $n$ vertices.
Then, for generic $\dd\in\RR^{|X_1|}$,
\[
{\rm rank}({\cal C}(X,\dd))=\nu(H_{X_1,X_k}).
\]
\end{conj}
 To see that this is indeed a special case, suppose that Conjecture~\ref{conj} is true for $d=n-1$. Let $G=(X_0,X_1)$ be the $1$-skeleton of $X$. Then, for a generic embedding $\p$ of $X_0$ in $\RR^{n-1}$, we have ${\rm rank}(R(G,\p))=|X_1|$, and therefore the image of $R(G,\p)$ is $\RR^{|X_1|}$. Combined with \eqref{chainrule}, we see that in this case
$$
{\rm rank}({\cal B}(X,\p))={\rm rank}({\cal C}(X,\dd)),$$
where $\dd=f_{X_1}(\p)$. 
Note that $f_{X_1}((\RR^d)^{n})$ is an open subset of $\RR^{|X_1|}$, and so in fact we have
$$
\max_{\p\in(\RR^d)^n}{\rm rank}({\cal B}(X,\p))
=
\max_{\dd\in\RR^{|X_1|}}{\rm rank}({\cal C}(X,\dd)).
$$
In other words, the generic rank of ${\cal C}(X,\dd)$ is equal to the generic rank of ${\cal B}(X,\p)$. 
Finally, the claim follows from Conjecture \ref{conj}, noting again that, since $d=n-1$, every $E\subset X_1$ is independent in the standard $d$-rigidity matroid.

Last, let us mention a relation between the $(k,d)$-volume rigidity matrix $\mathcal{B}(X,\p)$ studied here and a similar matrix studied by Lee in \cite{lee1996pl}. Let $L(X,\p)$ be the $|X_k|\times d|X_{k-1}|$ matrix defined by\footnote{The matrix $L(X,\p)$ is denoted in \cite{lee1996pl} by $R$ (see  \cite[p. 405]{lee1996pl}).}
\begin{equation}\label{eq:lee}
L(X,\p)_{\sigma,\tau}=\begin{cases}
        h_{\sigma,\tau} & \tau\subset \sigma,\\
        0 & \text{otherwise,}
    \end{cases}
\end{equation}
for every $\sigma\in X_k$ and $\tau\in X_{k-1}$, where for $\tau\subset \sigma$, denoting the unique vertex in $\sigma\setminus\tau$ by $v$, 
$h_{\sigma,\tau}\in\mathbb{R}^d$ is the altitude vector of the simplex $\p(\sigma)$ with respect to $\p(v)$ (pointing from $\p(\tau)$ to $\p(v)$).  
Note that, for $\tau=\sigma\setminus\{v\}$,  $h_{\sigma,\tau}=({\rm vol}_{k}(\sigma) k!/{\rm vol}_{k-1}(\tau)) N_{\sigma,v}$. 
It is then not hard to check, using \eqref{eq:lee}, \eqref{Bexplicit}, and the fact that for every $k$-simplex $\sigma$ we have
\[
    \sum_{v\in \sigma} {\rm vol}_{k-1}(\sigma\setminus \{v\}) N_{\sigma,v}=0,
\]
that
\[
    \mathcal{B}(X,\p)= -\frac{2}{(k!)^2}  L(X,\p) \cdot D(X,\p)\cdot P(X),
\]
where $D(X,p)\in \mathbb{R}^{|X_{k-1}|\times |X_{k-1}|}$ is the diagonal matrix defined by $
    D(X,\p)_{\tau,\tau}= {\rm vol}_{k-1}(\tau)^2
$ 
for all $\tau\in X_{k-1}$, and $P(X)$ is the $d|X_{k-1}|\times d|X_0|$ matrix where the $d\times d$ block indexed by $(\tau,v)$ equals $I_d$ (the $d\times d$ identity matrix) if $v\in \tau$ and $0$ otherwise, for every $\tau\in X_{k-1}$ and $v\in X_0$. We do not pursue this relation further in this work.

\begin{remark}
    Let us mention that the notion of $(k,d)$-volume rigidity was recently and independently introduced by James Cruickshank, Bill Jackson and Shin-ichi Tanigawa in \cite{cruickshank2025volume}. 
    In particular, our Proposition 
    \ref{main_part2} was independently proven in \cite[Theorem 5]{cruickshank2025volume}, by different techniques.
    We thank Bill, James and Shin-ichi for sharing a draft of their preprint just before both articles were submitted to arXiv, and for suggesting a better name for Lemma 3.
\end{remark}

\bibliographystyle{abbrv}
\bibliography{biblio}

\end{document}